\documentclass[12pt]{article}

\usepackage{tabularx}
\usepackage{tikz}
\usepackage{relsize}
mm \textheight=8.9in

\usepackage{amssymb}
\usepackage{amsmath}
\usepackage{amsthm}
\usepackage{color}

\newcommand{\br}[3]{{$#1$}$\lower4pt\hbox{$\tp\atop\raise4pt \hbox{$\scriptscriptstyle{#2}$}$} ${$#3$}}
\newcommand{\tw}[3]{{$#1$}${\,\scriptscriptstyle {#2}}\atop\raise9pt\hbox{$\scriptstyle\tp$} ${$#3$}}
\newcommand{\ttps}[2]{{#1}\raise5pt\hbox{$\lower12pt\hbox{$\scriptstyle\tp$}\atop \lower0pt\hbox{$\tilde\;$}$}\raise4.5pt\hbox{${\scriptstyle{#2}}$}}
\newcommand{\st}[1]{\mbox{${\,\scriptscriptstyle {#1}}\atop\raise5.5pt\hbox{$*$}$}}

\newcommand{\rd}[1]{\mbox{${\,\scriptscriptstyle {#1}}\atop\raise5.5pt\hbox{$\bullet$}$}}
\newcommand{\rt}[1]{\otimes_\chi}
\newcommand{\lt}[1]{\mbox{${\,\scriptscriptstyle {#1}}\atop\raise5.5pt\hbox{$\ltimes$}$}}
\newcommand{\btr}{\raise1.2pt\hbox{$\scriptstyle\blacktriangleright$}\hspace{2pt}}
\newcommand{\btl}{\raise1.2pt\hbox{$\scriptstyle\blacktriangleleft$}\hspace{2pt}}

\newcommand{\lcr}{\raise1.0pt \hbox{${\scriptstyle\rightharpoonup}$}}
\newcommand{\rcr}{\raise1.0pt \hbox{${\scriptstyle\leftharpoonup}$}}

\newcommand{\ttp}{{\lower12pt\hbox{$\tp$}\atop \hbox{$\tilde\;$}}}

\newcommand{\id}{\mathrm{id}}

\newcommand{\Bc}{\mathcal{B}}

\newcommand{\Hc}{\mathcal{H}}

\newcommand{\Ru}{\mathcal{R}}

\newcommand{\Uc}{\mathcal{U}}

\newcommand{\Q}{\mathcal{Q}}

\newcommand{\C}{\mathbb{C}}
\newcommand{\Z}{\mathbb{Z}}

\newcommand{\N}{\mathbb{N}}

\newcommand{\tp}{\otimes}

\newcommand{\U}{U}

\newcommand{\Fc}{\mathcal{F}}
\newcommand{\ve}{\varepsilon}

\newcommand{\op}{\oplus}
\newcommand{\la}{\lambda}

\newcommand{\End}{\mathrm{End}}

\newcommand{\Span}{\mathrm{Span}}

\newcommand{\Tr}{\mathrm{Tr}}

\newcommand{\Rm}{\mathrm{R}}

\newcommand{\diag}{\mathrm{diag}}

\newcommand{\ad}{\mathrm{ad}}
\newcommand{\Ad}{\mathrm{Ad}}

\newcommand{\g}{\mathfrak{g}}

\renewcommand{\k}{\mathfrak{k}}
\newcommand{\h}{\mathfrak{h}}

\newcommand{\s}{\mathfrak{s}}

\renewcommand{\o}{\mathfrak{o}}

\newcommand{\eps}{\epsilon}

\newcommand{\nn}{\nonumber}
\newcommand{\p}{\mathfrak{p}}
\renewcommand{\l}{\mathfrak{l}}

\newcommand{\al}{\alpha}

\newcommand{\bt}{\beta}

\newcommand{\be}{\begin{eqnarray}}
\newcommand{\ee}{\end{eqnarray}}

\newtheorem{thm}{Theorem}[section]
\newtheorem{propn}[thm]{Proposition}
\newtheorem{lemma}[thm]{Lemma}

\newtheorem{definition}[thm]{Definition}

\newcount\prg

\newcommand{\parag}{\advance\prg by1 {\noindent\bf\thesection.\the\prg\hspace{6pt}}}

\begin{document}
\title{Quantum symmetric conjugacy classes of non-exceptional groups}

\author{
Dakhilallah Algethami${}^{\dag,\ddag}$ and Andrey Mudrov${}^{\dag,\sharp}$
\vspace{10pt}\\
\small ${\dag}$ University of Leicester, \\
\small University Road,
LE1 7RH Leicester, UK,
\vspace{10pt}\\
\small ${\ddag}$ University of Bisha, \\
\small
255, Al Nakhil,
67714 Bisha, Saudi Arabia,
\vspace{10pt}\\
\small
${\sharp}$ Moscow Institute of Physics and Technology,\\
\small
9 Institutskiy per., Dolgoprudny, Moscow Region,
141701, Russia,
\vspace{10pt}\\
\small
 e-mail: daaa3@leicester.ac.uk,  am405@le.ac.uk
}
\date{ }

\maketitle

\begin{abstract}
We evaluate one-dimensional representations of quantum symmetric conjugacy classes of classical matrix  groups
along with their quantum stabilizer subgroups.
\end{abstract}
\begin{center}
\end{center}
\\
{\small \underline{AMS classification codes}: 17B10, 17B37, 53D55.}
\newpage

\section{Introduction}

\pagenumbering{arabic}
For a geometric space $X$ equipped with a left action of a group $G$, every point $a\in X$ generates an orbit $Ga\subset X$.
Such an orbit is a homogeneous $G$-space isomorphic to $G/K$ where $K\subset G$ is the isotropy subgroup of $a$.
Thus the point $a$ defines an equivariant homomorphism
$$ \iota_a\colon \Fc(X)\hookrightarrow \Fc(G/K)\simeq \Fc(G)^K\subset \Fc(G)$$
of  function algebras realizing $\Fc(Ga)$ as a subalgebra of $K$-invariants in $\Fc(G)$. Algebraically the $G$-action is dualized
as a coaction $\delta\colon \Fc(X)\to \Fc(G)\tp \Fc(X)$.  The point $a$ defines a
$K$-invariant character $\chi_a\colon \Fc(X)\to \C$ acting by evaluation $\chi_a\colon f\mapsto f(a)$ for  $f\in \Fc(X)$. The  homomorphism $\iota_a$
factorizes to   the composition $(\id\tp \chi_a)\circ \delta$. This simple observation underlies a quantization method  of Poisson-Lie $G$-spaces based on q-analogs of
classical points, \cite{DM0}. Although this approach is not the most general,
it has many attractive features because of its quasi-classical nature.

A restriction to the method  is that a quantized function algebra $\Fc_q(X)$ has typically very few
one-dimensional representations, if any. In the classical limit, they turn to  points where the Poisson bivector vanishes.
If such a point is quantized to a character of $\Fc_q(X)$,
then $\Fc_q(G/K)$ can be realized as a subalgebra in $\Fc_q(G)$ similarly to the classical case.
In this paper we are concerned with the case when $X$ is the group space $G$ with the conjugation action on itself, and the orbits are conjugacy classes of $G$.
We are interested in the classes of elements whose square is in the center of $G$. They are isomorphic to complexified symmetric spaces \cite{Hel}.

A more general approach to quantization that  is special to closed conjugacy classes in a simple algebraic group $G$ does not rely on quantum points.
The quantized polynomial ring $\C_q[O]$ of a semi-simple class $O\subset G$ is formulated in terms of generators and relations, as a quotient of a  quantized
affine ring  $\C_q[G]$ by an invariant ideal that is the kernel of a certain representation. Note that $\C_q[G]$
  is not a Hopf algebra dual to quantized universal enveloping of the Lie algebra $\g=\mathrm{Lie}(G)$; it is a different quantization of the coordinate ring $\C[G]$.

It is a relatively rare case when the two approaches can be applied simultaneously.
Fortunately symmetric conjugacy classes do have quantum points.
We construct them  for  groups from the four infinite series, along with their quantum
isotropy subgroups in the total quantum group. This way we match the two  approaches: via generalized parabolic Verma modules and quantum symmetric pairs
\cite{Let}.

To be  specific, let $G$ be the complex general linear, orthogonal, or symplectic  Poisson group relative to the standard Drinfeld-Sklyanin Poisson bracket determined
by the standard classical r-matrix
and let $U_q(\g)$ be the corresponding Drinfeld-Jimbo quantum group \cite{D,ChP}.
There is another Poisson bracket on $G$ that makes it    a Poisson-Lie manifold under the conjugation action \cite{STS}.
It is also cooked up from the same classical r-matrix but in a different way. Quantization of $\C[G]$ along this bracket is
a $U_q(\g)$-module algebra $\C_q[G]$ that is related with the so-called reflection equation, see (\ref{RE}) below.
The Poisson bivector on $G$ is tangent to every conjugacy class making it a Poisson-Lie manifold over $G$, \cite{AlM}.

 Equivariant quantization of semi-simple conjugacy  classes with respect to an action of $U_q(\g)$
is constructed in
\cite{M1, M2} via a representation  of $\C_q[G]$ by linear operators on certain $U_q(\g)$-modules of highest weight (pseudo-parabolic Verma modules).
The quotient of $\C_q[G]$ by the kernel of the representation is a quantized polynomial ring of the class.
 This way  quantum conjugacy classes are   treated as subvarieties of the quantum algebraic group $G$.
If symmetric, they have one-dimensional representations of our interest.

Let  $\Ru$ be a universal R-matrix of $U_q(\g)$ and $V$ the defining (natural) fundamental representation.
Denote by $\Q$ the image of  the "exponential split-Casimir" $\Ru_{21}\Ru$
in $\End(V)\tp U_q(\g)$.
Denote by  $R$ the   image of $\Ru$ in $\End(V)\tp \End(V)$ and
by  $\varpi$ the invariant projector onto the trivial one-dimensional sub-representation in $V\tp V$ for orthogonal and symplectic $\g$.
The algebra $\C_q[G]$ of (the connected component of) a  classical  matrix group $G\subset \End(\C^N)$
is generated by $\Q_{ij}$, $i,j=1,\ldots,N$, modulo a $U_q(\g)$-invariant ideal, which
validates  reflection equation (RE) \cite{KSS}
\begin{equation}
R_{21}\Q_1 R_{12}\Q_2=\Q_2 R_{21}\Q_1 R_{12}.
\nn
\end{equation}
This is a deformation of commutativity of matrix coordinate functions.

For the orthogonal and symplectic quantum groups,  $\Q$ also satisfies
\begin{equation}
\Q_2 S_{12} \Q_2 \varpi = \ve q^{-N+\ve}\varpi =\varpi\Q_2 S_{12} \Q_2,
\nn 
\end{equation}
where $\ve=+1$ if $\g$ is  orthogonal and $\ve=-1$ if $\g$ is symplectic. The braid matrix $S$ is the "quantum permutation" $PR$, where $P$ is the ordinary
flip of tensor factors in  $V\tp V$. For the special quantum linear group one has to assume
a unit quantum determinant but that will be automatically done once we impose other conditions of the class:
the eigenvalues of its matrices along with multiplicities, which are class  invariants.
 That will be done   by fixing roots of  the minimal polynomial on $\Q$
and values of q-traces, see below.

The relations on the entries of the matrix $\Q$ define an associative algebra of function of conjugacy classes $\C_q[O]$.
 Technically, finding one-dimensional representations  of this algebra  means that the entries $\Q_{ij}$ are sent to numbers
 producing a numerical matrix $A$ satisfying all the relations of  $\C_q[O]$.

Remark that there are various types of  reflection equation in the literature.
So, a different version was considered  \cite{Nou}. There is no obvious relation
of that version to the  one  considered in this paper although there is a certain correspondence amongst the solutions.
It is the present version of RE  that is quantizing the Poisson structure of concern.

 In the classical limit $q\rightarrow 1$, the matrix $A$ turns to a matrix  $A_0\in O\subset G$ where the Poisson bracket vanishes.
 Let $K\subset G$ denote its    centralizer subgroup and consider $O$ as the coset space $G/K$.
  Then the function algebra   $\C[G/K]$ is a subalgebra of $K$-invariants in $\C[G]$.
  On the other hand, $O$ is a subvariety in $G$, so the function algebra   $\C[O]$ is a quotient algebra of $\C[G]$, that yields  a different realization of the same space.

Although the projection $\C_q[G]\to \C_q[O]$ does always exist in contrast with the embedding $\C_q[G/K]\subset \C_q[G]$,
the quantum stabilizer subgroup $\C_q[K]$ is very often absent.
 However if $\C_q[O]$ has  a quantum point, the universal enveloping algebra $U(\k)$  centralizing   $A_0$
 is quantizable as a coideal subalgebra $U_q(\k)\subset U_q(\g)$. The general theory of
 coideal subalgebras was developed by G. Letzter in \cite{Let}. We specialize that  construction
  to symmetric conjugacy classes  and present  expressions for  generators of $U_q(\k)$.

\section{Basics of quantum groups}
In this paper  $\g$ is a simple matrix Lie algebra of classical type  with a fixed Cartan subalgebra $\h\subset \g$
represented by diagonal matrices.

Fix an invariant inner product on $\g$, restrict it to $\h$ and transfer it to $\h^*$ by duality.
The root system is expressed through an orthonormal  basis $\{\ve_i\}_{i=1}^n\subset  \h^*$  if $\g$ is orthogonal or symplectic and $\{\ve_i\}_{i=1}^{n+1}$
in the case of $\g=\s\l(n+1)$.
A basis $\Pi$ of  simple positive roots is
$$
\Pi =\{\al_i=\ve_i-\ve_{i+1}:1\leq i\leq n\},\quad \g=\s\l(n+1),
$$
$$
\Pi =\{\al_i=\ve_i-\ve_{i+1}:1\leq i<n\}\cup\{\al_n=\ve_n\},\quad \g=\s\o(2n+1),
$$
$$
\Pi=\{\al_i=\ve_i-\ve_{i+1}:1\leq i<n\}\cup\{\al_n=\ve_{n-1}+\ve_n\},\quad \g=\s\o(2n),
$$
$$
\Pi=\{\al_i=\ve_i-\ve_{i+1}:1\leq i<n\}\cup\{\al_n=2\ve_n\},\quad \g=\s\p(2n).
$$

A total ordering on a set of positive roots  $ \Rm^+ \supset\Pi$ of $\g$ is called  normal if every positive root  split in the sum $\mu +\nu$  with $\mu, \nu \in \Rm^+$ is between $\mu$ and $\nu$.

We assume that a non-zero scalar deformation parameter $q\in \C$  is  not a root of unity.
The quantum group $U_q(\g)$ is a $\C$-algebra generated by $e_{\al_{i}}=e_i$, $f_{\al_{i}}=f_i$, and $q^{\pm h_{\al_i}}=q^{\pm h_i}$ subject to the following relations, \cite{ChP}.

\begin{enumerate}

\item [(i)]	$q^{h_{i} }  e_{j }  q^{-h_{i} }=q^{(\alpha_i,\alpha_j ) } e_{j }$,
\item [(ii)]$q^{h_{i} }  f_{j }  q^{-h_{i }}=q^{-(\alpha_i,\alpha_j ) } f_{j }$,
\item [(iiia)] $ e_{i }f_{j }-f_{j }e_{i }=\delta_{ij}\frac{q^{h_{i} }-q^{-h_{i} }}{q-q^{-1}}$, unless $\g=\s\p(2n)$ and $i=n$,
\item [(iiib)] $ e_{n }f_{j }-f_{j }e_{n }=\delta_{ij}\frac{q^{h_{n} }-q^{-h_{n} }}{q^2-q^{-2}}$, if $\g=\s\p(2n)$,
\item [(iv)] the generators $q^{\pm h_i}$  commute with each other and satisfy $q^{-h_i}q^{h_i}=q^{ h_i}q^{- h_i}=1$,
\item [(v)]  the generators $e_i,\ f_i$ are subject to the quantum q-Serre relations whose exact formulation can be found in \cite{ChP}.
\end{enumerate}
Note that our normalization of  the negative generators simplifies the basic representation assignment in Section \ref{SecIsotropy}.
This results in a difference  from the standard definition by the commutator $[e_n,f_n]=\frac{q^{h_{n} }-q^{-h_{n} }}{q^\frac{1}{2}-q^{-\frac{1}{2}}}$
for $\g=\s\o(2n+1)$.

 We use shortcuts  $[z]_q=\frac{q^z-q^{-z}}{q-q^{-1}}$ for $z\in \h+\C$, and $[x,y]_a=xy-ayx$, where $x, y \in U_q(\g)$, and $a$ is a scalar.

A Hopf algebra structure is fixed by  comultiplication on the generators as
$$\Delta(e_i)= q^{h_i }\otimes e_i+e_i\otimes 1,\quad \Delta(f_i)=f_i\otimes q^{-h_i }+1\otimes f_i,\quad \Delta(q^{\pm h_i })=q^{\pm h_i}\otimes q^{\pm h_i}.$$
The counit $\eps$ is a homomorphism $U_q(\g)\to \C$ that is zero on all $e_i$, $f_i$ and returns $1$ on $q^{\pm h_i}$.
The antipode $\gamma$ can be readily evaluated on the generators from the comultiplication.

Let  $\{v_i\}_{i=1}^N$ be the
standard weight basis of the natural $U_q(\g)$-module $\C^N$ of minimal dimensions.
Each $v_i$ carries the weight $\ve_i$, under the  convention $\ve_{i'}=-\ve_i$, where we assume  $i'=N-i+1$,
if $\g$ is symplectic or orthogonal.

By $R \in \mathrm{End}(\C^N\otimes \C^N)$ we denote the image of  the universal R-matrix of $U_q(\g)$ in the natural representations.
Up to a scalar multiplier, it is expanded in  the standard matrix basis $\{e_{ij}\}_{i,j=1}^N\subset  \End(\C^N)$ as  in \cite{FRT}:
\begin{equation}
R=q\sum_{i=1}^Ne_{ii}\otimes e_{ii}+\sum\limits_{\substack{i,j=1\\ i\neq j }}^Ne_{ii}\otimes e_{jj}+(q-q^{-1})\sum\limits_{\substack{i,j=1\\ i<j }}^Ne_{ji}\otimes e_{ij},
\label{R-mat-1}
\end{equation}
for   $\g=\s\l(N)$ and
\begin{equation}
R=\sum_{i,j=1}^N q^{\delta_{ij}-\delta_{ij'}}e_{ii}\otimes e_{jj} +(q-q^{-1})\sum\limits_{\substack{j,i=1\\ j<i}}^{N} (e_{ij}\otimes e_{ji} -\kappa_i\kappa_jq^{\rho_i-\rho_j} e_{ij}\otimes e_{i'j'}),
\label{R-mat}
\end{equation}
for symplectic and orthogonal $\g$. Here
 $\rho$ is the half-sum of all positive roots
 and $\rho_{i}=(\rho,\ve_i)$.
All $\kappa_j$ equal $1$  if $\g=\s\o(N)$. In the case of  $\g=\s\p(N)$,    $\kappa_j=1$ when $j\leq \frac{N}{2}$ and $\kappa_j=-1$ otherwise.

Define a matrix
$S = P R$, where  $P\colon \C^N\tp \C^N\to \C^N\tp \C^N$ is the permutation of tensor factors, $P (v\otimes w)=w\otimes v, \ v,w\in \C^N$.
It  satisfies the braid identity$$S_{12}S_{23}S_{12} = S_{23}S_{12}S_{23},$$
with $S_{12} = S \otimes 1$ and $S_{23} = 1 \otimes S$.

A  matrix   $A\in \mathrm{End}(\C^N )$ is said to be
a solution of the (numerical) reflection equation if
\begin{equation}
SA_2SA_2 = A_2SA_2S.
\label{4}
\end{equation}
It is then a one-dimensional representation, $\Q_{ij}\mapsto A_{ij}$, of a quantum matrix algebra whose entries $\Q_{ij}$ are subject to the relation (\ref{RE}).
We will seek for $A$ that also satisfies  the relation  (\ref{OC}) (in the orthogonal and symplectic case) as well as  other identities of quantum symmetric conjugacy classes
listed in the next section.

\section{Points in quantum
symmetric conjugacy classes }
\subsection{Symmetric conjugacy classes and their quantization}
\label{secQuantSymCl}
We call a semisimple conjugacy class $O$ of a simple complex algebraic group $G$ symmetric if
conjugation $\Ad_s$ with $s\in O$ is an involution, that is,  if  $s^2$ is in the center of $G$. Such a class is  a complexification of
a Riemannian symmetric space.

We will categorize symmetric classes of a particular group by the order of its elements, $\ell$, which family we denote by T$\ell$.
It is sufficient to consider only classes of types T2 and T4 because the only case of $SL(N)$ when $\ell\not =2$ divides $2N$ can be readily reduced to T2.

The type T2 comprises AIII, BDI, and CII in the Cartan classification \cite{Hel}.
Classes from T4 include $SO(N)/GL\left(\frac{N}{2}\right)$ and $SP(N)/GL\left(\frac{N}{2}\right)$ (DIII and respectively CI)  with even $N$.

Quantization of the polynomial ring on a conjugacy class passing through a point $t$ in the maximal torus $T\subset G$
can be faithfully represented as a subalgebra of linear endomorphisms of an irreducible  $U_q(\g)$-module of highest weight that depends on $t$.
It is called base weight and it is constructed as follows.

Let $\k\subset \g$ be the Lie subalgebra centralizing $t$. The triangular decomposition of $\g$
induces triangular decomposition of $\k=\k_-\op \h\op \k_+$ with $\k_\pm \subset \g_\pm$.
The root system of $\k$ is naturally identified with a root subsystem in $\Rm=\Rm_\g$, which we denote by $\Rm_\k$,
so that $\Rm^+_\k\subset\Rm^+_\g$ (but generally  $\Pi_\k\not \subset \Pi_\g$).

Set $\kappa$ equal to the half-sum of positive roots of $\k$, $\kappa=\frac{1}{2}\sum_{\al\in \Rm^+_\k}\al$.
Define a base weight $\la\in \h^*$    by the assignment
$$
q^{\la}\colon
q^{h_\al}\mapsto q^{(\la,\al)}= \pm \sqrt{\al(t)} q^{(\kappa-\rho,\al)},\quad  \al \in \Pi,
$$
where  a root $\al$ is regarded as a multiplicative character  of the maximal torus, $\al\colon T\to \C$, and
the signs are chosen arbitrarily.

The Verma module $V_\la$ has  submodules
$V_{\la-\al}$ of highest weight $\la-\al$ for each  $\al\in \Pi_\k$. The quotient $M_\la=V_\la/\sum_{\al \in \Pi_\k}V_{\al-\la}$ called base module
 is irreducible for all but may be  a finite number of values of $q$ away from a root of unity.

Let $V$ be the natural representation of $U_q(\g)$ and $\Q$ be the image of $\Ru_{21}\Ru$ in $\End(V)\tp U_q(\g)$.
The subalgebra in $U_q(\g)$ generated by the matrix entries $\Q_{ij}$ is a quantization of the coordinate
ring $\C[G]$ denoted by $\C_q[G]$. It is an $\ad$-invariant subalgebra in $U_q(\g)$ and its image   in $\End(M_\la)$
is an equivariant quantization of the coordinate ring $\C[O]$, where $O$ is the conjugacy class of $t$.

Different choices of the base weight $\la$ corresponding to the same $t$ (as well as different $t$ from the same conjugacy class) yield different base modules but isomorphic quantizations.
The quantized coordinate ring $\C_q[O]$ can be described as a quotient of $\C_q[G]$ by an invariant ideal,
the annihilator of $M_\la$.
This ideal is generated by the entries of the minimal polynomial of $\Q$ of degree 2 and the values of  $q$-trace $\Tr_q(\Q)$ (for
symmetric $O$). They are determined by the initial point $t$ via the base weight $\la$ and depend only on the class $O$.

Below we recall an explicit description of symmetric  classes of connected non-exceptional groups.
Let us fix the following initial points for classes of type T2:
\be
t&=&
\left\{
\begin{array}{cl}
\diag(-1,\ldots,-1,1,\ldots,1)&  \g=\s\l(N),
\\
 \diag(-1,\ldots,-1,1,\ldots,1,-1,\ldots,-1), &  \g=\s\o(N),\>\s\p(N),
\end{array}
\right.
\label{IP1}
\ee
assuming the diagonal symmetric under the inversion $t_{ii}\mapsto t_{i'i'}$ in the symplectic and orthogonal case.
Let the number of $+1$-s equal $P$ and the number of $-1$-s equal $M=N-P \in 2\N$.
Note that the approach of  \cite{M5} delivers quantization of classes only in  connected $G$.

For even orthogonal and symplectic $G$
there are also two Hermitian symmetric classes of type T4 that pass through points
\be
t&=&
\diag(i,\ldots, i,- i,\ldots,- i),
\label{IP2}
\ee
where $i=\sqrt{-1}$. The numbers of $\pm i$-s are the same and equal to $\frac{N}{2}$.

Symmetric conjugacy classes of non-exceptional groups are generated by entries of a matrix $\Q$
subject to
\begin{equation}
R_{21}\Q_1 R_{12}\Q_2=\Q_2 R_{21}\Q_1 R_{12}.
\label{RE}
\end{equation}
For $G$ of types $B,C,D$ they also satisfy the condition
\be
\Q_2 S_{12} \Q_2 \varpi = \ve q^{-N+\ve}\varpi =\varpi\Q_2 S_{12} \Q_2
\label{OC}
\ee
with $\ve=+1$ for the orthogonal groups and $\ve=-1$ for symplectic.

The classes of type T2 satisfy  relations
\begin{equation}
(\Q+q^{-P})( \Q-q^{-M})=0,
\label{IE}
\end{equation}
\begin{equation}
\Tr_q(\Q)=[P]_q-[M]_q,\quad G=SL(N),
\label{TRSL}
\end{equation}
\begin{equation}
\Tr_q(\Q)=[P-1]_q-[M-1]_q,\quad G=O(N),
\label{TRO}
\end{equation}
\begin{equation}
\Tr_q(\Q)=[P+1]_q-[M+1]_q,\quad G=SP(N).
  \label{TRS}
\end{equation}
A general symmetric class  in $SL(N)$ has eigenvalues $\pm e^{\frac{\pi i k}{N}}$ with fixed $k=0,\ldots, N-1$ and the number of minuses equal to $M$.
Its quantization is described by similar formulas as of the type  T2,  where the roots of the minimal polynomial  of $\Q$ and $\Tr_q(\Q)$  should be multiplied
by $e^{\frac{\pi i k}{N}}$.

Classes of type T4 satisfy the relations
\be
(\Q-i q^{-\frac{N}{2}+\ve})( \Q+  i q^{-\frac{N}{2}+\ve})=0,
\label{IE_H}
\ee
\begin{equation}
\Tr_q(\Q)=0.
\label{TROS_H }
\end{equation}
Remark that specializing the value of the (quantum) determinant is redundant because it is fixed automatically by the roots of the minimal
polynomial of $G$.

The above equations follow from  decomposition of the module $\C^N\tp M_\la$
into a direct sum of two irreducible submodules with highest weights $\ve_1+\la$ and $\ve_{M/2}+\la$.
They are  $\Q$-eigenspaces  of eigenvalues $-q^{-P}$ and, respectively, $q^{-M}$,
for type T2 and $\pm iq^{-\frac{N}{2}+\ve}$ for type T4. The $q$-traces are calculated via the formula $\Tr_q(\Q)=\Tr(q^{2h_\la+2h_\rho})$, where the right-hand side
is the ordinary trace in $\End(\C^N)$, see \cite{M4} for details.

As we already noted, quantization via generalized parabolic Verma modules   covers only those conjugacy class\-es which are in the connected component
of the identity. The method of quantum initial points allows to bypass this restriction.

\subsection{Quasi-classical limit of reflection equation}
\label{SecAnsatz}

Let $G$ be a simple complex Lie group and   $\g$ its Lie algebra.
Denote by $T_g(G)$ the tangent space at  $g\in G$ and by $T(G)$ the tangent vector  bundle.
Every element $\xi\in \g$ generates left and right invariant vector fields, $\xi^l_g=g\xi \in  T(G)$, $\xi^r_g=\xi g\in  T(G)$
(in a matrix representation), and an adjoint vector field $\xi^{\ad}_g=\xi^l_g-\xi^r_g$.

Let $\omega\in S^2(\g)$ be the symmetric invariant, unique up to a scalar multiplier.
Suppose that $\varrho\in \g\wedge \g$ is such that  $\mathrm{r}=\varrho+\omega$ satisfies the classical Yang-Baxter equation
$$
[\mathrm{r}_{12},\mathrm{r}_{13}]+[\mathrm{r}_{12},\mathrm{r}_{23}]+[\mathrm{r}_{13},\mathrm{r}_{23}]=0.
$$
The RE Poisson bracket on the group is given by the bivector field
\be
\varrho^{\ad,\ad}+\omega^{r,l}-\omega^{l,r} \in \wedge^2T(G),
\label{RE-Poisson}
\ee
where
the indices $r,l$ designate the corresponding vector fields generated by  tensor factors.
\begin{lemma}
The bivector field $\varphi=\omega^{l,r}-\omega^{r,l}\in \wedge^2 T(G)$ is invariant.
\end{lemma}
\begin{proof}
Denote by $g\mapsto \Ad_g$ the adjoint representation on $\g$.
For each $g\in G$, identification of  $T_g(G)$  with $\g$ by right translations, $\g g\simeq T_g(G)$,
  takes  the bivector $\varphi_g $ to  $(\Ad_g\tp \id)(\omega)-(\id \tp \Ad_g)(\omega)$.
The conjugation
 $a\colon g\mapsto a g a^{-1}$ induces  an operator   on $\wedge^2 T_g(G)$ that sends $\varphi_g$ to
$$
(\Ad_{ag}\tp \Ad_a)(\omega)-(\Ad_a \tp \Ad_{ag})(\omega)=(\Ad_{aga^{-1}}\tp \id)(\omega)-(\id \tp \Ad_{aga^{-1}})(\omega)=\varphi_{aga^{-1}},
$$
because $\omega$ is $\Ad$-invariant.
\end{proof}
The next lemma implies that the Poisson bracket simplifies on symmetric conjugacy classes.
\begin{lemma}
  The bivector field $\omega^{l,r}-\omega^{r,l}\in \wedge^2 T(G)$ turns zero at every $g\in G$ such that $\Ad_g^2=\id$.
\end{lemma}
\begin{proof}
Since $\Ad_g^2=\id$, we get
$$
\varphi_g=(\Ad_g\tp \id ) (\omega)-(\id \tp \Ad_g) (\omega)
=(\Ad_g\tp \id ) (\omega)-(\Ad_g^2 \tp \Ad_g) (\omega) =0,
$$
by the invariance of $\omega$.
\end{proof}
\noindent
It follows from this lemma that the bivector field generated by $\varrho$ is Poisson. This is a special case of Proposition 2.1 from \cite{DS}.

Now fix a triangular decomposition $\g=\g_-\op \h\op \g_+$ with the Cartan subalgebra $\h$ and assume
that $\varrho$ is the standard classical skew symmetric r-matrix. It is  the anti-symmetrized
inverse Killing pairing $\g_-\tp \g_+ \to \C$,
$$
\varrho=\sum_{\al\in \Rm^+} e_\al \wedge f_\al, \quad (e_\al,f_\al)=1, \quad \forall \al \in \Rm^+.
$$
Denote by  $T\subset G$   the maximal torus that corresponds to $\h$.
Suppose that $a\in N(T)\subset G$, the normalizer of the torus and let $\theta$ be the corresponding element of the Weyl group, $\theta^2=\id$.

Conjugation with $a$ takes the root space $\g_\al$   to   $\g_{\theta(\al)}$.
Denote by $\k$ the subalgebra of $\theta$-fixed points and by $\l\subset \k$ the semisimple subalgebra
generated by simple root vectors $e_\al,f_\al$ with $\theta(\al)=\al$.

Suppose that $\Rm$ splits to a disjoint union $\Rm_\l\cup \Phi^+\cup \Phi^-$, where $\Phi^\pm\subset \Rm^\pm$ and $\theta \Phi^+=\Phi^-$.
We call such points $a\in G$ admissible.
Then $-\theta$ preserves $\Phi^+$ and we denote $\tilde \al=-\theta(\al)$ for all $\al\in \Phi_+$.
Every symmetric conjugacy class has an admissible point, \cite{Di}, Lemma 11.1.16.
\begin{propn}
\label{class_points}
The Poisson bivector (\ref{RE-Poisson}) is vanishing at admissible $a$.
\end{propn}
\begin{proof}
 The $\omega$-term vanishes in view of the previous lemma, so we need to prove that $\varrho^{\ad,\ad}$ turns zero at $a$.

Let $x_\al\in \C$ be defined from the equality  $\Ad_a e_\al=x_\al f_{\tilde \al}$   for $\al\in \Phi^+$.
The operator $\Ad_a\in \End(\h^*)$ is orthogonal, then
\be
x_\al=(\Ad_a e_\al, e_{\tilde \al})=(\Ad_a e_\al, \Ad_a^2 e_{\tilde \al})=(e_\al, \Ad_a e_{\tilde \al})=x_{\tilde \al}
\label{al-tilde-al}
\ee
for all such $\al$. We have also $\Ad_a f_\al=y_\al e_{\tilde \al}$    for $\al\in \Phi^+$, and $x_\al y_{\tilde \al}=1$ because  $\Ad_a$ is an involution.
On identification of $T_a(G)$ with $\g$ via right shift by $a$ we get $\xi_a^{\ad}=\Ad_a(\xi)-\xi$.
Then
$$
\varrho^{\ad,\ad}_a=\sum_{\al\in \Phi^+} (x_\al f_{\tilde \al}-e_\al)\wedge (y_\al e_{\tilde \al}-f_\al),
$$
because the sum over $\al\in \Rm_\l^+$ vanishes.
This bivector turns zero  provided
$$
\sum_{\al\in \Phi^+} y_\al e_\al \wedge e_{\tilde \al}=0, \quad
  \sum_{\al\in \Phi^+}x_\al f_{\tilde \al}\wedge f_\al=0,
\quad
 \sum_{\al\in \Phi^+} e_\al \wedge f_\al + \sum_{\al\in \Phi^+} x_\al y_\al f_{\tilde \al}\wedge  e_{\tilde \al}=0.
$$
These equalities  are fulfilled if, respectively
$$
x_{\al}=x_{\tilde \al}, \quad y_{\al}=y_{\tilde \al}, \quad x_{\al}y_{ \al}=1\quad \forall \al \in \Phi^+.
$$
They hold true indeed thanks to (\ref{al-tilde-al}) and to  $x_\al y_{\tilde \al}=1$.
\end{proof}

Next we list  matrices satisfying the conditions of Proposition \ref{class_points}.
 \begin{equation}
A_0=\pm(\sum_{m<i<m'}e_{ii}+\sum_{1\leqslant i,i'\leqslant m} y_ie_{ii'}), \quad m\in \N, \quad m\leqslant \frac{N}{2},
\label{cso2}
\end{equation}
with $y_{i'}y_{i}=1$ for  $O(N)$- and  $SL(N)$-classes  of type T2;
 \begin{equation}
A_0= \sum_{i=1}^{N}y_ie_{ii'}, \quad N\in 2\N ,
\label{csp4}
\end{equation}
with  $y_{i'}y_{i}=-1$ for $SP(N)$-classes  of type T4;
\begin{equation}
A_0=\pm\bigl(\sum_{m<i<m'}e_{ii}+\sum\limits_{\substack{1\leq i,i'\leq m\\ i\in 2\Z +1 }}z_i(e_{i, i'-1}-e_{i+1,i'})\bigr)
, \quad m\in 2\N, \quad m\leqslant \frac{N}{2},
\label{csp2}
\end{equation}
with $z_iz_{i'-1}=1$, for  $SP(N)$-classes of type T2;
\begin{equation}
A_0=\sum\limits_{\substack{i=1\\ i\in 2\Z +1 }}^{N-1}z_i(e_{i, i'-1}-e_{i+1,i'}),\quad N\in 2\N,
\label{cso4}
\end{equation}
with $z_iz_{i'-1}=-1$ for $O(N)$-classes of type T4. For odd $n=\frac{N}{2}$ the matrix contains a diagonal term
 $\pm i(e_{nn}-e_{n'n'})$. Note that we do not restrict $A_0$ to the identity component of $G$.

\begin{definition}
We call a matrix $A$ with  $\lim_{q\to 1} A=A_0$ a quantization of $A_0$ if  it satisfies (\ref{RE}) and, where appropriate, (\ref{OC})
on replacement $\Q\to A$.
 \end{definition}
We will search for a quantization of $A_0$ in   the following two shapes:
\be
A=\sum_{i=1}^{N}x_{i}e_{ii}+\sum_{i=1}^{N}y_{i}e_{ii'},
\quad
A=\sum_{i=1}^Nx_ie_{ii}+\sum\limits_{\substack{i=1\\  i\in 2\Z +1 }}^N  z_i(e_{i, i'-1}-e_{i+1,i'}).
\label{ansatz}
\ee
In other words, the off-diagonal part of $A$ is either skew diagonal or,  for $N\in 2\N$, skew block-diagonal with the blocks being
scalar multiple of $\nu=\begin{bmatrix}
1&0\\
0&-1
\end{bmatrix}
$.

\begin{thm}
\label{RE-mat}
The following matrices quantize  classical points (\ref{cso2}-\ref{cso4}):
\be
 \pm\bigl(\sum_{i=1}^{m} q^{-m}(1-q^{-N+2m}) e_{ii}+\sum_{m<i<m'} q^{- m} e_{ii}+\sum_{1\leq i,i'\leq m}y_{i}e_{ii'}\bigr),
 \> \g =\s\l(N), \s\o(N),&&
\label{so(N)}
\\
 \pm\bigl(\sum_{i=1}^m  q^{-m}(1-q^{-N+2m}) e_{ii}+\sum_{m<i<m'} q^{- m}e_{ii}+\sum\limits_{\substack{1\leq i,i'\leq m\\ i\in 2\Z +1 }}z_i(e_{i, i'-1}-e_{i+1,i'})\bigr),
\>\g=\s\p(N),&&
\label{sp(N)}
\ee
with a convention  on the integer $m$ as above and
\be
&\sum_{i=1}^Ny_ie_{ii'},\quad \g=\s\p(N),
\label{sp(N)-Herm}
\\
& \sum\limits_{\substack{i=1\\ i\in 2\Z +1 }}^{N-1}z_i(e_{i, i'-1}-e_{i+1,i'}),\quad \g=\s\o(N),
\label{so(N)-Herm}
\ee
where the parameters $y_i,z_i \in \C$ satisfy the conditions
\begin{itemize}
  \item $y_{i}y_{i'}=q^{-N}$ in (\ref{so(N)}) and $y_iy_{i'}=-q^{-N-2}$ in (\ref{sp(N)-Herm}),
  \item  $z_{i}z_{i'-1}=q^{-N}$ for all odd $i$ in  (\ref{sp(N)}),
  \item $z_{i}z_{i'-1}=-q^{-N+2}$ for odd $i\leqslant \frac{N}{2}$  in (\ref{so(N)-Herm}).
\end{itemize}
\end{thm}
\begin{proof}
The proof is  a direct verification of the identities (\ref{RE}) and (\ref{OC}) where appropriate,  within the chosen ansatz  (\ref{ansatz}).
The case of $\g=\s\l(N)$
was studied  in \cite{M0}.
\end{proof}
These solutions are matrices of the shapes
\begin{center}
\begin{picture}(100,100)
\put(-5,-5){\line(0,1){110}}
\put(110,-5){\line(0,1){110}}
\put(75,30){\line(0,1){40}}
\put(35,30){\line(0,1){40}}
\multiput(15,95)(3,-3){3}{\circle*{1}}
\multiput(52,52)(3,-3){3}{\circle*{1}}
\multiput(95,95)(-3,-3){3}{\circle*{1}}
\multiput(8,8)(3,3){3}{\circle*{1}}

\put(100,100){$\scriptstyle y_{1}$}

\put(0,100){$\scriptstyle \la+\mu$}
\put(20,80){$\scriptstyle \la+\mu$}

\put(37,62){$\scriptstyle \la$}

\put(68,32){$\scriptstyle \la$}

\put(80,80){$\scriptstyle y_{m}$}

\put(20,20){$\scriptstyle y_{m'}$}
\put(0,0){$\scriptstyle y_{1'}$}

\end{picture}
\quad,\quad\quad\quad\quad\quad
\begin{picture}(100,100)
\put(-5,-5){\line(0,1){110}}
\put(110,-5){\line(0,1){110}}
\put(75,30){\line(0,1){40}}
\put(35,30){\line(0,1){40}}
\multiput(15,95)(3,-3){3}{\circle*{1}}
\multiput(52,52)(3,-3){3}{\circle*{1}}
\multiput(95,95)(-3,-3){3}{\circle*{1}}
\multiput(10,8)(3,3){3}{\circle*{1}}

\put(95,100){$\scriptstyle z_{1}\nu$}

\put(0,100){$\scriptstyle \la+\mu$}
\put(20,80){$\scriptstyle \la+\mu$}

\put(37,62){$\scriptstyle \la$}

\put(68,32){$\scriptstyle \la$}

\put(75,80){$\scriptstyle z_{m-1}\nu$}

\put(20,20){$\scriptstyle z_{m'}\nu$}
\put(0,0){$\scriptstyle z_{2'}\nu$}

\end{picture}
\quad,
\end{center}
where  the central diagonal block is of size $N-2m$ and $\la$ and $\mu$ are their eigenvalues.
In the case of (\ref{so(N)},\ref{sp(N)})
they are  $\la= \pm q^{-m}$ and $\mu=\mp q^{-N+m}$, of multiplicities $N-m$ and $m$, respectively.
In the case of (\ref{so(N)-Herm}) and (\ref{sp(N)-Herm}) the  central block disappears and the eigenvalues are
$\pm i q^{-\frac{N}{2}+\ve}$.
\begin{propn}
Suppose that $\det \lim_{q\to 1}A=1$. Then $A$ satisfies the equations of a quantum  conjugacy class
of $A_0=\lim_{q\to 1}A$ from Section \ref{secQuantSymCl}.
\end{propn}
\begin{proof}
The  minimal polynomial of  $A$ has degree 2, and its eigenvalues are $-q^{-P}$ ad $q^{-M}$, as in (\ref{IE})
if one sets $P=N-m$, $M=m$ for the common factor $+1$, and $P=m$ and $M=N-m$ for   $-1$.
The minimal polynomials of $A$ and $\Q$ are the same, cf.    (\ref{IE_H}).

The proof then reduces to calculation of   $\Tr_q(A)$.
It easy to see that $\Tr_q(A)=0$   for $A$ from (\ref{sp(N)-Herm}) and  (\ref{so(N)-Herm}), as required.
With the above convention on $m$, the quantum trace of $A$ from (\ref{so(N)}) and (\ref{sp(N)}) is equal
to $\Tr_q(\Q)$  in  (\ref{TRSL}, \ref{TRO}) and (\ref{TRS}) respectively.
\end{proof}

\section{Left  coideal subalgebra $U_q(\k)\subset U_q(\g)$}
In this section we describe the quantum stabilizer subgroups of quantum points in symmetric conjugacy class, following \cite{Let}.
We prove basic facts here for reader's convenience.

Let $K\subset G$ be the  centraliser subgroup of $A_0=\lim_{q\to 1}A\in G$ and  $\k$ its Lie algebra. Conjugation with $A_0$
is an involutive automorphism $\theta$ of $G$ for which $K$ is the subgroup of fixed points.
It preserves the maximal torus of $T$ and   induces an automorphism  of the root system which we denote by the same letter $\theta$.
The Lie subalgebra in $\k$ generated by simple root vectors whose roots are fixed by $\theta$ is denoted by $\l$, as in the previous section.

For each $\al\in \Rm^+$ we denote by  $\g^\al$  a simple Lie subalgebra in $\g$ whose root system is generated
by simple roots entering $\al$ with non-zero coefficients. Let $\Rm_{\g^\al}$ be its (irreducible) root system and $\Pi_{\g^\al}\subset \Pi$
the basis of simple roots. Set $\l^\al=\l\cap \g^\al$.

Recall  that the root system splits to a disjoint union
$\Rm=\Rm_\l \cup\Phi^-\cup \Phi^+$
where   $\Phi^\pm \subset \Rm^\pm$ are  subsets flipped by $\theta$.

\begin{propn}
\label{loc-norm-ord}
Let $\Pi=\Pi_\l\cup \bar \Pi_\l$ be the decomposition of basis with  $\bar \Pi_\l\subset \Phi^+$.
Fix $\al\in \bar \Pi_\l$ and set $\tilde \al =-\theta(\al)\in \Rm^+$. Then
\begin{itemize}
  \item  $\tilde \al \in \al'+ \Rm^+_\l$ with $\al'\in \bar \Pi^+_\l$,
  \item  there is a normal order on $\Rm^+_{\g^{\tilde\al}}$ where roots from  $\Rm_{\l^{\tilde\al}}^+$ are on the left and
  $\tilde \al$ is the next root to the right of $\Rm_{\l^{\tilde\al}}^+$.
\end{itemize}
\end{propn}
\begin{proof}
To prove the first assertion, consider $-\theta$ as an involutive operator on the vector space $V=\h^*/\h^*_\l$.
Here $\h_\l^*$ is the vector subspace generated by $\Pi_\l$.
We have $V=V^+\op V^-$, where $V^+ =\Span \{\bar \Pi_\l \}\mod \h^*_\l $. The roots $\al\in \bar \Pi_\l$ descend to a basis in $V^+$. Restriction of $-\theta$ to
$V^+$ is an involutive matrix with non-negative entries. This implies that it has
only one non-zero entry equal to $1$ in every row and every column. In other words, this operator induces an involutive permutation $\al\mapsto \al'$
on $\bar \Pi_{\l}$.

  There is a normal order  on $\Rm^+_{\g^{\tilde\al}}$ where roots from  $\Rm_{\l^{\tilde\al}}^+$ are on the left.
  The corresponding factorization of the longest element of the Weyl group reads $w_{\g^{\tilde \al}}=w_{\l^{\tilde\al}} u$,
  where $w_{\l^{\tilde\al}}$ is the longest element of the Weyl group of $\Rm_{\l^{\tilde\al}}$. The leftmost factor in $u$ is the simple
  reflection relative to $\al'$.
  Then the next root to the right after $\Rm_{\l^{\tilde\al}}^+$ is $w_{\l^{\tilde\al}}(\al')\in \al' + \Rm_{\l^{\tilde\al}}^+$. We argue that
  it equals $\tilde \al=-\theta(\al)$. It is sufficient to show that
  $\tilde \al$ is the highest weight in the irreducible  $\l^{\tilde \al}$-submodule in $\g^{\tilde \al}$  generated by $e_{\al'}$.
  But if $\eta =\tilde \al + \bt \in \Rm_{\l^{\tilde\al}}^+$ is a root with $\bt\in \Z_+\Pi_{\l^{\tilde\al}}$, then
  $-\theta(\eta)=\al - \bt\in \Rm^+$, which forces $\bt=0$.
\end{proof}
Remark that $\al'\not =\al$ correspond to an arc in the corresponding Satake diagram connecting $\al$ and $\al'$, cf. \cite{GW}.

For $A_0$ of type T2 we have
\begin{equation*}
\theta(\ve_i)=\begin{cases}
\ve_{i'},& \mathrm{if}\ i\leq m, \\
\ve_{i}, &\mathrm{if}\ i>m,\\
\end{cases},
\quad \g=\s\l(N),
\quad\theta(\ve_i)=\begin{cases}
-\ve_i,&\mathrm{if}\ i\leq m, \\
\ve_{i}, &\mathrm{if}\ i>m,\\
\end{cases},
\quad \g= \s\o(N),
\end{equation*}

\begin{equation*}
\theta(\ve_i)=\begin{cases}
-\ve_{i+1},& \mathrm{if}\ i<m\ \mathrm{is\ odd}, \\
-\ve_{i-1},& \mathrm{if}\ i\leq m\ \mathrm{is\ even}, \\
\ve_{i}, & \mathrm{if}\ i>m,\\
\end{cases}
\quad \g= \s\p(N).
\end{equation*}
Then $\Pi_\l$ is explicitly
$$
\{\al_i\}_{i=m+1}^{n-m}, \quad \{\al_i\}_{i=m+1<n}^n,
\quad
\{\al_{2i+1} \}_{i=0}^{\frac{m}{2}-1}\cup  \{\al_i\}_{i=m+1}^n
$$
for $\g=\s\l(N)$, $\g=\s\o(N)$, and  $\g=\s\p(N)$, respectively (recall that $m$ is even for $\s\p(N)$).

For simple roots  $\al \in\bar \Pi_\l$ we have
$$
\tilde \al_i=\al_{n+1-i}, \quad i=1,\ldots, m-1,\quad  \tilde \al_i=\al_{n+1-i},\quad i=n+2-m,\ldots, n,
$$
$$
 \tilde \al_m=\sum_{l=m+1}^{n+1-m}\al_l, \quad \tilde \al_{n+1-m}=\sum_{l=m}^{n-m}\al_l,
$$
for $\g=\s\l(n+1)$,
$$ \tilde \al_i=\al_i, \quad i=1,\ldots, m-1,\quad  \tilde \al_m= \al_m+2\sum_{l=m+1}^{n}\al_l,$$
for $\g=\s\o(2n+1)$,
$$ \tilde \al_i=\al_i, \quad i=1,\ldots, m-1,\quad  \tilde \al_m= \al_m+2\sum_{l=m+1}^{n-2}\al_l+\al_{n-1}+\al_n,\quad m\leqslant n-2,$$
$$  \tilde \al_{n-1}= \al_n,\quad \tilde \al_{n}= \al_{n-1},\quad m=n-1,\quad \tilde \al_i=\al_i, \quad i=1,\ldots, n,\quad m=n,$$
for $\g=\s\o(2n)$,
$$ \tilde \al_{2i}=\sum_{l=2i-1}^{2i+1}\al_{l}, \quad i=1,\ldots, \frac{m}{2}-1,\quad  \tilde \al_m= \al_{m-1}+\al_{m}+2\sum_{l=m+1}^{n-1}\al_l+\al_n,\quad m<n,
$$
$$
 \tilde \al_m=2\al_{m-1}+\al_m, \quad m=n, \quad n\in 2\Z,
$$
for $\g=\s\p(2n)$. We have $\al_i'=\al_{n+1-i}$ and  $\al_{n+1-i}'=\al_i$ for $\g=\s\l(n+1)$ and $i\leqslant m<\frac{N}{2}$.
For $\s=\s\o(2n)$ and $m=n-1$, $\al'_n=\al_{n-1}$ and $\al'_{n-1}=\al_n$.
In  all other cases,  $\al'=\al$.

For $A_0$ of type T4 we have
\begin{equation*}
\theta(\ve_i)=
-\ve_i,\>  i=1,\ldots, n  \\
,
\> \g=\s\p(2n), \quad
\theta(\ve_i)=\begin{cases}
-\ve_{i+1},&  n>i \in 2\N+1, \\
-\ve_{i-1},&   n\geq i\in 2\N, \\
\ve_{i},&  n= i\in 2\N+1, \\\end{cases}
\> \g=  \s\o(2n).
\end{equation*}

Then $\Pi_\l=
\{\varnothing \},
$
for $\g=\s\p(2n)$ and
$\Pi_\l=\{\al_{2i+1} \}_{i=0}^{l-1}$
for $\g=\s\o(2n)$.
On other roots, $\theta$ acts by
$$ \tilde \al_i=\al_i, \quad i=1,\ldots, n,$$
for $\g=\s\p(2n)$ and by
$$ \tilde \al_{i}=\sum_{l=i-1}^{i+1}\al_{l}, \quad   2\N \ni i<n-1,\quad \tilde \al_{i}=\al_{i}, \quad i=n\in 2\N,$$
$$ \tilde \al_{n-1}=\al_{n-2}+\al_{n},\quad \tilde \al_{n}=\al_{n-2}+\al_{n-1},\quad n\in 2\N+1.$$
for $\g=\s\o(2n)$.
Here we have $\al'_{n-1}=\al_{n}$ and $\al'_{n}=\al_{n-1}$ and $\al'=\al$ for all other $\al$.

The Lie algebra $\k$ is generated by $\l$ and elements $X_\al=e_\al+c_\al f_{\tilde \al}$, $h_{\tilde \al}-h_\al$ with $\al\in \bar \Pi_\l$,
where $c_\al$ depend on
parameters of the classical point $A_0$.
We restrain from calculating $c_\al$ because we give  their quantum deformations in the next section.

The algebra $U(\k)$ is quantized as a coideal subalgebra, $U_q(\k)$, which makes a  quantum symmetric pair together with  $U_q(\g)$.
We specialize the general theory of quantum symmetric pairs developed in \cite{Let} to the case of conjugacy classes and give a short proof of basic facts for reader's convenience.

Recall that a left   coideal algebra $\Bc$ of a Hopf algebra $\Hc$  satisfies the inclusion
 $$\Delta(\Bc)\subset \Hc\tp \Bc.$$
For instance, $U_q(\g_+)$  is a left coideal subalgebras in $U_q(\g)$. More examples   come from the following fact.
Denote by $F_\al=q^{h_\al}f_\al$ and by $\Uc^-\subset U_q(\g)$ the subalgebra generated by $F_\al$.
Choose a normal order on $\Rm^+$ and extend $F_\bt$ to  all $\bt \in \Rm^+$ as $a_\bt e_{-\bt}$ in \cite{KT}, where
$a_\bt$ is a non-zero complex number. Denote by $\Uc_m^-$ the subalgebra in $\Uc^-$ generated by $F_{\bt^i}$ with $i\leqslant m$.
\begin{lemma}
\label{coideal_norm_order}
For each for each $m=1,\ldots,  |\Rm^+|$, $\Uc_m^-$ is a left coideal subalgebra in $\U_q(\g)$.
Furthermore,
$$
\Delta(F_{\bt^m})=  (F_{\bt^m}\tp 1+q^{h_{\bt^m}}\tp F_{\bt^m}) +  U_q(\g)\tp \Uc^-_{m-1}.
$$
\end{lemma}

\begin{proof}
This readily follows from  \cite{KT}, Prop. 8.3. upon the assignment $F_\al\mapsto e_{-\al}$ on simple root vectors, extended as an algebra isomorphism.
It is then a coalgebra anti-isomorphism.
\end{proof}
We conclude, by   Lemma \ref{coideal_norm_order},
that for each $\al\in \bar \Pi_\l$ the elements $F_{\tilde \al}$ and $F_\mu$ with  $\mu\in \Pi_{\l^{\tilde\al}}$
form a left   coideal subalgebra in $U_q(\g)$.

A subalgebra $U_q(\k)\subset U_q(\g)$ is generated over  $U_q(\l)$ by
$$
X_\al=q^{h_{\tilde \al}-h_\al}e_{\al}+c_\al F_{\tilde \al},
$$
 and  by $q^{\pm(h_{\tilde\al}- h_\al)}$ with  $\al \in \bar \Pi_\l$
and  non-zero $c_\al\in \C$ to be determined in the next section.

\begin{thm}[\cite{Let}]
  The algebra generated by $e_\al,f_\al, q^{\pm h_\al}$ with  $\al \in \Pi^\theta$, and $X_\al=q^{h_{\tilde \al}-h_\al}e_\al+c_\al F_{\tilde \al}$,
  $q^{\pm(h_{\tilde \al}-h_\al)}$ with $\al \in \bar\Pi^\theta $, is a left coideal   in $U_q(\g)$.
\end{thm}
\begin{proof}
By Lemma \ref{coideal_norm_order}, one has  for each   $\al \in \Pi^\theta$:
$$
\Delta(F_{\tilde \al})\in  (q^{h_{\tilde\al}}\tp F_{\tilde\al}+F_{\tilde \al}\tp 1) +  U_q(\g)\tp U_q(\l^{\tilde \al}),
$$
$$
\Delta (q^{h_{\tilde \al}-h_\al}e_\al)= q^{h_{\tilde \al}}\tp q^{h_{\tilde \al}-h_\al}e_\al+q^{h_{\tilde \al}-h_\al}e_\al \tp q^{h_{\tilde \al}-h_\al}.
$$
Adding the lines together one gets
$$
\Delta(f_{\tilde \al})\in q^{h_{\tilde\al}} \tp X_{\al}+ F_{\tilde\al}\tp 1  + q^{h_{\tilde \al}-h_\al}e_\al\tp  q^{h_\al -h_{\tilde \al}}  +   U_q(\g)\tp  U_q(\l^{\tilde \al}),
$$
as required.
\end{proof}

\section{$U_q(\k)$ as a quantum isotropy subgroup}
\label{SecIsotropy}
In this section we specify the mixture parameters $c_\al$,  $\al \in \bar \Pi^\theta$, entering the definition of $U_q(\k)$  to make $A$ a fixed point of  $U_q(\k)$.
Consider the  natural representation $\pi:U_q(\g) \rightarrow \End(\C^N)$ determined by the assignment
\begin{equation*}
q^{h_i} \mapsto\begin{cases}
\sum_{j=1}^{N}q^{ \delta_{ji}-\delta_{j(i+1)}} e_{jj},&\mathrm{if}\ i\leq n,\  \g=\s\l(n+1),\\
\sum_{j=1}^{N}q^{ \delta_{ji}-\delta_{ji'}} q^{- \delta_{j(i+1)}+\delta_{j(i+1)'}}e_{jj},&\mathrm{if}\ i< n,\  \g=\s\o(N),\ \mathrm{or}\ \g=\s\p(N),\\
\sum_{j=1}^{N}q^{\delta_{jn}-\delta_{jn'}}e_{jj},&\mathrm{if}\ i=n ,\ \g=\s\o(2n+1), \\
\sum_{j=1}^{N}q^{2\delta_{jn}-2\delta_{jn'}}e_{jj},&\mathrm{if}\ i=n ,\ \g=\s\p(2n), \\
\sum_{j=1}^{N}q^{\delta_{j(n-1)}-\delta_{j(n-1)'}} q^{ \delta_{jn}-\delta_{jn'}}e_{jj},&\mathrm{if}\ i=n ,\ \g=\s\o(2n), \\

     \end{cases}
\end{equation*}

\begin{equation*}
e_i\mapsto\begin{cases}
e_{i,i+1},&\mathrm{if}\ i\leq n  ,\ \g=\s\l(n+1),\\
e_{i,i+1}-e_{(i+1)',i'},& \begin{cases}\mathrm{if}\ i\leq n, \  \g=\s\o(2n+1), \\
\mathrm{if}\ i<n, \ \g=\s\o(2n),\ \mathrm{or}\  \g=\s\p(2n),\end{cases}\\
e_{n-1,n+1}-e_{(n+1)',(n-1)'},&\mathrm{if}\ i=n ,\ \g=\s\o(2n), \\
e_{n,n+1},&\mathrm{if}\ i=n  ,\ \g=\s\p(2n),\\

     \end{cases}
\end{equation*}
and the similar assignment  to  $f_i$ as for $e_i$ with  $e_{ij}$  changed  to $e_{ji}$.
This representation is compatible  with the R-matrices (\ref{R-mat-1}) and (\ref{R-mat}).

Consider  a matrix $A\in \End(V)$  from Theorem \ref{RE-mat} and the involution $\theta$ determined by its classical limit $A_0=\lim_{q\to 1} A$.
Let  $U_q(\l)\subset U_q(\k)$ be the subalgebras associated with $\theta$ as in the previous section.
It is straightforward to see that the matrix $A$ commutes with  $\pi(U_q(\l))$
and all $\pi(q^{h_{\tilde \al}-h_\al})$,   $\al \in \bar \Pi_\l$.
The parameters $c_\al$ will be fixed from the requirement that $A$ commutes
with all
$$
\pi(X_\al)=\pi(q^{h_{\tilde \al}-h_\al}e_{\al}+c_\al F_{\tilde \al}), \quad \al \in \bar \Pi_\l.
$$
We will present explicit expressions for $F_{\tilde \al}$ with non-simple $\tilde \al$, that is, when $\Pi_{\l^{\tilde\al}} \not =\varnothing$.

In order to make the presentation more readable, it is supplemented with the  Satake diagrams. Recall that
they are Dynkin diagrams with additional data that parameterize symmetric pairs. Simple roots from $\Pi_\l$ are depicted by black nodes;
they are fixed by $\theta$. White nodes label simple roots from $\bar \Pi_\l$. If the root $\al'\in \Pi_{\tilde \al}$ is distinct from $\al$, it is connected with $\al$
by an arc.
\subsection{$\k=\s\l(n+1)$}
There are two types of Satake diagrams describing  to symmetric classes of special linear group $SL(N)$, $N=n+1$:
 
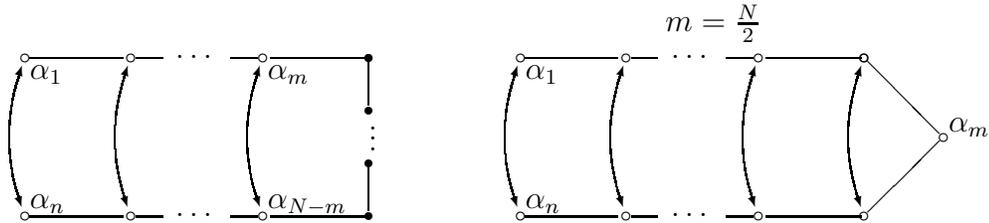
\begin{center}
\begin{picture}(160,80)
\put(10,10){\circle{3}}
\put(50,10){\circle{3}}
\put(100,10){\circle{3}}
\put(140,10){\circle*{3}}
\put(12,10){\line(1,0){35}}
\put(52,10){\line(1,0){10}}
\put(88,10){\line(1,0){10}}
\put(67,10){$\ldots$}
\put(103,10){\line(1,0){36}}

\put(10,70){\circle{3}}
\put(50,70){\circle{3}}
\put(100,70){\circle{3}}
\put(140,70){\circle*{3}}
\put(12,70){\line(1,0){35}}
\put(52,70){\line(1,0){10}}
\put(88,70){\line(1,0){10}}
\put(67,70){$\ldots$}
\put(103,70){\line(1,0){36}}

\put(140,12){\line(0,1){17}}
\put(140,30){\circle*{3}}
\put(140,35){$\vdots$}
\put(140,50){\circle*{3}}
\put(140,52){\line(0,1){17}}
\put(140,70){\circle*{3}}

 \qbezier(8,16)(0,40)(8 ,64) \put(6.9,19.5){\vector(1,-3){2}} \put(7.0,61.3){\vector(1,3){2}}
 \qbezier(48,16)(40,40)(48 ,64) \put(46.9,19.5){\vector(1,-3){2}} \put(47.0,61.3){\vector(1,3){2}}
\qbezier(98,16)(90,40)(98 ,64) \put(96.9,19.5){\vector(1,-3){2}} \put(97.0,61.3){\vector(1,3){2}}

\put(12,62){$\al_1$}
\put(102,62){$\al_m$}

\put(12,13){$\al_n$}
\put(102,13){$\al_{N-m}$}

 \end{picture}
\quad \quad
\begin{picture}(160,85)
\put(65,80){$m=\frac{N}{2}$}

\put(10,10){\circle{3}}
\put(50,10){\circle{3}}
\put(100,10){\circle{3}}
\put(140,10){\circle{3}}
\put(12,10){\line(1,0){35}}
\put(52,10){\line(1,0){10}}
\put(88,10){\line(1,0){10}}
\put(67,10){$\ldots$}
\put(103,10){\line(1,0){36}}

\put(10,70){\circle{3}}
\put(50,70){\circle{3}}
\put(100,70){\circle{3}}
\put(140,70){\circle{3}}
\put(12,70){\line(1,0){35}}
\put(52,70){\line(1,0){10}}
\put(88,70){\line(1,0){10}}
\put(67,70){$\ldots$}
\put(103,70){\line(1,0){36}}

\put(141.5,11.5){\line(1,1){27}}
 \put(170,40){\circle{3}}
\put(141.5,68.5){\line(1,-1){27}}
\put(140,70){\circle{3}}

 \qbezier(8,16)(0,40)(8 ,64) \put(6.9,19.5){\vector(1,-3){2}} \put(7.0,61.3){\vector(1,3){2}}
 \qbezier(48,16)(40,40)(48 ,64) \put(46.9,19.5){\vector(1,-3){2}} \put(47.0,61.3){\vector(1,3){2}}
\qbezier(98,16)(90,40)(98 ,64) \put(96.9,19.5){\vector(1,-3){2}} \put(97.0,61.3){\vector(1,3){2}}
\qbezier(138,16)(130,40)(138 ,64) \put(136.9,19.5){\vector(1,-3){2}} \put(137.0,61.3){\vector(1,3){2}}

\put(12,62){$\al_1$}
\put(172,42){$\al_m$}

\put(12,13){$\al_n$}
 \end{picture}

\end{center}
All roots $\tilde \al_i$ with $i<m$ and $i>N-m$ are simple. For such roots we find
$$
c_{\al_{i}}=
\frac{y_{i+1}}{y_{i}},\quad \mathrm{if}\quad i< m \quad \mathrm{or} \quad i> N-m.
$$
The diagram on the right corresponds to    $m=\frac{N}{2}$ for even $N$.
Then $\tilde \al_m =\al_m$, and
$$
c_{\al_{i}}=\frac{q^{-2m+1}}{y_m^2}.
$$
In the case of  $m<\frac{N}{2}$, the roots $\tilde \al_m$ and $\tilde \al_{N-m}$ are not simple.
We define the corresponding root vectors as
$$F_{\tilde \al_m}=[\dots[F_{m+1}, F_{m+2}]_{q} ,\dots  F_{n+1-m}]_{q}, \quad F_{\tilde \al_{N-m}}=[\dots[F_{m}, F_{m+1}]_{\bar q} ,\dots  F_{n-m}]_{\bar q}.$$
The mixture parameters are found to be
$$
c_{\al_{m}}=\frac{(-1)^{N+1}q^{-N+m}}{y_m}, \quad
c_{\al_{N-m}}=\frac{(-1)^{N+1}q^{2N - 5m - 3} }{y_m}.
$$
The algebra $U_q(\k)$ stabilizes the quantum point (\ref {so(N)}).
\subsection{$\k=\s\o(2n+1)$}
For the  odd orthogonal group, we have $\tilde \al_i=\al_i$ for all $i<m$.
Then
$$
c_{\al_{i}}=-\frac{q y_{i+1}}{y_{i}}.
$$
When $m\leqslant n-1$, we distinguish the following  two cases:
$$F_{\tilde \al_m}=[\dots[[[\dots[[F_{m}, F_{m+1}]_{q}, F_{m+2}]_{q},\dots  F_n]_{q},F_n],F_{n-1}]_{q},\dots F_{m+1}]_{q},\quad  m<n-1, $$
$$F_{\tilde \al_m}=[[F_{n-1}, F_{n}]_{q},F_n],\quad  m=n-1.$$
 These correspond to  the diagram on the left below:
 \begin{center}
\begin{picture}(200,30)

\put(2,10){\line(1,0){10}}
\put(38,10){\line(1,0){10}}
\put(0,10){\circle{3}}
\put(17,10){$\ldots$}
\put(50,10){\circle{3}}
\put(51.5,10){\line(1,0){27}}
\put(80,10){\circle{3}}
\put(81.5,10){\line(1,0){27}}
\put(110,10){\circle{3}}

\put(112,10){\line(1,0){10}}
\put(148,10){\line(1,0){10}}
\put(110,10){\circle*{3}}
\put(127,10){$\ldots$}
\put(160,10){\circle*{3}}
\put(161,8.5){\line(1,0){24.5}}
\put(161,11.5){\line(1,0){24.5}}
\put(181,7){$>$}
\put(190,10){\circle*{3}}

\put(82,14){$\al_m$}

\put(0,14){$\al_1$}

\put(187,14){$\al_n$}

 \end{picture}
\quad\quad
\begin{picture}(120,30)
\put(0,10){\circle{3}}
\put(1.5,10){\line(1,0){27}}
\put(32,10){\line(1,0){10}}
\put(68,10){\line(1,0){10}}
\put(30,10){\circle{3}}
\put(80,10){\circle{3}}
\put(47,10){$\ldots$}

\put(110,10){\circle{3}}

\put(81,8.5){\line(1,0){24.5}}
\put(81,11.5){\line(1,0){24.5}}
\put(101,7){$>$}

\put(107,14){$\al_m$}

\put(0,14){$\al_1$}

 \end{picture}
\end{center}
In the situation of the right diagram with $m=n$, the  root $\tilde \al_m$ equals $\al_n$.
The mixture coefficient are
$$
c_{\al_{m}}=
\frac{(-1)^{(n - m+1)}}{y_m q^{m + 1}}, \quad  m<n,\quad c_{\al_{m}}=
\frac{-1 }{y_m q^{m}}, \quad  m=n.
$$
The algebra $U_q(\k)$ stabilizes the quantum point (\ref {so(N)}).
\subsection{$\k=\s\p(2n)$}
Consider first the symmetric class of order 4 corresponding to the Hermitian symmetric space.
It is described by the Satake diagram
 \begin{center}
 \begin{picture}(120,30)
\put(0,10){\circle{3}}
\put(1.5,10){\line(1,0){27}}
\put(32,10){\line(1,0){10}}
\put(68,10){\line(1,0){10}}
\put(30,10){\circle{3}}
\put(80,10){\circle{3}}
\put(47,10){$\ldots$}

\put(110,10){\circle{3}}

\put(84,8.5){\line(1,0){24.5}}
\put(84,11.5){\line(1,0){24.5}}
\put(80,7){$<$}

\put(107,14){$\al_n$}

\put(0,14){$\al_1$}

 \end{picture}

\end{center}
For all simple roots, one has $\al=\tilde \al$, and the mixture parameters are
$$
c_{\al_{i}}=-q\frac{y_{i+1}}{y_{i}}, \quad i=1,\ldots, n-1, \quad c_{\al_n}=-\frac{1 }{y_n^2q^{2n}}.
$$
This  $U_q(\k)$ is the stabilizer of (\ref{sp(N)-Herm}).

For non-Hermitian classes (of order 2) the parameter $m$ is even. Then
$$F_{\tilde \al_{2i}}=[[F_{2i}, F_{2i+1}]_{q}, F_{2i-1}]_{q}\ \mathrm{with}\ i = 1,\ldots, \frac{m}{2}-1, \quad \mbox{and }$$
$$
c_{\al_{2i}}=
-\frac{qz_{2i+1}}{z_{2i-1}},\quad \mathrm{if}\  2i<m.
$$
The case of $m\leqslant n-1$ corresponds to the Satake diagrams
 \begin{center}
\begin{picture}(230,30)

\put(0,10){\circle*{3}}
\put(1.5,10){\line(1,0){27}}
\put(30,10){\circle{3}}
\put(31.5,10){\line(1,0){27}}
\put(60,10){\circle*{3}}

\put(62,10){\line(1,0){10}}
\put(77,10){$\ldots$}
\put(98,10){\line(1,0){10}}

\put(110,10){\circle{3}}
\put(111.5,10){\line(1,0){27}}
\put(140,10){\circle*{3}}

\put(142,10){\line(1,0){10}}
\put(157,10){$\ldots$}
\put(178,10){\line(1,0){10}}

\put(190,10){\circle*{3}}
\put(194,8.5){\line(1,0){24.5}}
\put(194,11.5){\line(1,0){24.5}}
\put(190,7){$<$}
\put(220,10){\circle*{3}}

\put(0,14){$\al_1$}
\put(107,14){$\al_m$}
\put(217,14){$\al_n$}

 \end{picture}
\end{center}
Depending on the value of $m$ we define the root vectors
$$F_{\tilde \al_m}=[[\dots[[\dots[[F_{m}, F_{m+1}]_{q}, F_{m+2}]_{q},\dots F_n]_{q^{2}},F_{n-1}]_{q},\dots F_{m+1}]_{q},F_{m-1}]_{q},\quad \mathrm{if}\ m<n-1,$$
$$F_{\tilde \al_{m}}=[[F_{n-1}, F_{n}]_{q^2}, F_{n-2}]_{q},\quad \mathrm{if}\ m=n-1\in 2\Z,$$
with the mixture parameter
$$
c_{\al_{m}}=
 \frac{(-1)^{n+1}}{z_{m-1} q^{m}} , \quad \mathrm{if}\ m\leqslant n-1,\quad
 $$
in both cases.
The situation $m=n$ for even $n$  by encoded in the Satake diagram
 \begin{center}
\begin{picture}(180,30)

\put(0,10){\circle*{3}}
\put(1.5,10){\line(1,0){27}}
\put(30,10){\circle{3}}
\put(31.5,10){\line(1,0){27}}
\put(62,10){\line(1,0){10}}
\put(98,10){\line(1,0){10}}
\put(60,10){\circle*{3}}
\put(110,10){\circle{3}}
\put(77,10){$\ldots$}

\put(111.5,10){\line(1,0){27}}

\put(140,10){\circle*{3}}

\put(144,8.5){\line(1,0){24.5}}
\put(144,11.5){\line(1,0){24.5}}
\put(140,7){$<$}
\put(170,10){\circle{3}}

\put(2,14){$\al_1$}
\put(152,14){$\al_n=\al_m$}


 \end{picture}
\end{center}
The root vector and the mixture parameter are then
$$F_{\tilde \al_n}=[[F_{n}, F_{n-1}]_{q^2}, F_{n-1}], \quad   \mbox{with}\quad
 c_{\al_{n}}=-\frac{1}{(q^2 + 1)q^{2n - 2}z^2_{n-1}}.
$$
The subalgebra   $U_q(\k)$ constructed this way is the stabilizer of (\ref{sp(N)}).

\subsection{$\k=\s\o(2n)$}
There are three Satake diagrams describing even orthogonal symmetric  classes of order 2.
In all cases $\tilde \al_i=\al_i$ if $i<m$, and
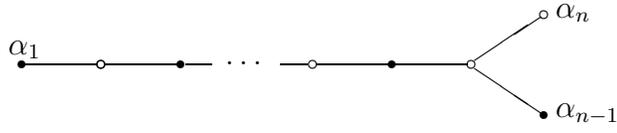
$$
c_{\al_{i}}=
-\frac{qy_{i+1}}{y_{i}}
$$
for such $i$. So we have to consider the case $i=m$.

For $m\geqslant n-1$ the root $\tilde \al_m$ is simple.
The case  $m=n$ corresponds to the diagram
\begin{center}
\begin{picture}(300,60)

\put(0,30){\circle{3}}
\put(1.5,30){\line(1,0){27}}

\put(32,30){\line(1,0){10}}
\put(68,30){\line(1,0){10}}
\put(30,30){\circle{3}}
\put(47,30){$\ldots$}
\put(80,30){\circle{3}}
\put(81.5,31.5){\line(3,2){25}}
\put(81.5,28.5){\line(3,-2){25}}
\put(107.5,49){\circle{3}}
\put(107.5,11){\circle{3}}

\put(0,34){$\al_1$}
\put(112,8){$\al_{n-1}$}
\put(112,48){$\al_m=\al_n$}

\put(170,27){$\mbox{with} \quad c_{\al_{m}}=
-\frac{1}{y_{n-1} y_n q^{2n - 1}}$.}

\end{picture}
\quad
\end{center}
The case $m=n-1$ corresponds to the diagram
\begin{center}
 \begin{picture}(300,50)

\put(0,30){\circle{3}}
\put(1.5,30){\line(1,0){27}}

\put(32,30){\line(1,0){10}}
\put(68,30){\line(1,0){10}}
\put(30,30){\circle{3}}
\put(47,30){$\ldots$}
\put(80,30){\circle{3}}
\put(81.5,31.5){\line(3,2){25}}
\put(81.5,28.5){\line(3,-2){25}}
\put(107.5,49){\circle{3}}
\put(107.5,11){\circle{3}}
\qbezier(110,16)(115,30)(110 ,44) \put(111,19){\vector(-1,-3){2}} \put(111,41){\vector(-1,3){2}}

\put(0,34){$\al_1$}
\put(112,8){$\al_m=\al_{n-1}$}
\put(112,48){$\al_n$}

\put(170,27){$\mbox{with} \quad
c_{\tilde \al_m}=-\frac{1}{y_{i-1}q^n}$.}
 \end{picture}
\end{center}
The case $m\leqslant n-2$ is described by the diagram
 \begin{center}
\begin{picture}(200,40)

\put(2,20){\line(1,0){10}}
\put(38,20){\line(1,0){10}}
\put(0,20){\circle{3}}
\put(17,20){$\ldots$}
\put(50,20){\circle{3}}
\put(51.5,20){\line(1,0){27}}
\put(80,20){\circle{3}}
\put(81.5,20){\line(1,0){27}}

\put(112,20){\line(1,0){10}}
\put(148,20){\line(1,0){10}}
\put(110,20){\circle*{3}}
\put(127,20){$\ldots$}
\put(160,20){\circle*{3}}
\put(161.5,21.5){\line(3,2){25}}
\put(161.5,18.5){\line(3,-2){25}}
\put(187.5,39){\circle*{3}}
\put(187.5,1){\circle*{3}}

\put(0,24){$\al_1$}
\put(80,24){$\al_m$}

\put(192,0){$\al_{n-1}$}
\put(192,38){$\al_n$}

 \end{picture}
\end{center}
We define
$$F_{\tilde \al_m}=[\dots[[\dots[[F_{m}, F_{m+1}]_{q}, F_{m+2}]_{q},\dots F_n]_{q},F_{n-2}]_{q},\dots F_{m+1}]_{q}
\quad \mbox{for}\quad m<n-2 \quad\mbox{and}$$
$$F_{\tilde \al_m}=[[F_{n-2}, F_{n-1}]_{q}, F_{n}]_{q}\quad \mbox{for}\quad  m=n-2.$$
In all cases,
$$
c_{\al_{i}}=\frac{(-1)^{n - m}}{y_m q^{m + 1}}.
$$
This completes the description of coideal subalgebras for non-Hermitian symmetric classes. Constructed this way,
$U_q(\k)$  leaves the quantum point  (\ref{so(N)})   fixed.

The Hermitian classes of $SO(2n)$ fall into two families depending on whether $n$ is even or odd.
In all cases we define
$$F_{\tilde \al_i}=[[F_{i}, F_{i+1}]_{ q}, F_{i-1}]_{q},\quad c_{\al_{i}}=
-q\frac{z_{i+1}}{z_{i-1}},
 $$
for  even $i < n-1$.
For even $n$ we have the diagram
  \begin{center}
\begin{picture}(220,40)

 \put(20,20){\circle*{3}}
\put(21.5,20){\line(1,0){27}}
\put(50,20){\circle{3}}
\put(51.5,20){\line(1,0){27}}

\put(82,20){\line(1,0){10}}
\put(118,20){\line(1,0){10}}
\put(80,20){\circle*{3}}
\put(97,20){$\ldots$}

\put(50,20){\circle{3}}
\put(51.5,20){\line(1,0){27}}

\put(130,20){\circle{3}}
\put(131.5,20){\line(1,0){27}}

\put(160,20){\circle*{3}}
\put(161.5,20){\line(1,0){27}}

\put(190,20){\circle{3}}
\put(191.5,21.5){\line(3,2){25}}
\put(191.5,18.5){\line(3,-2){25}}
\put(217.5,39){\circle{3}}
\put(217.5,1){\circle*{3}}

\put(15,24){$\al_1$}

\put(222,0){$\al_{n-1}$}
\put(222,38){$\al_n$}

 \end{picture}
\end{center}
The root $\tilde \al_n=\al_n$ is simple. The mixture parameter is
$
c_{\al_{n}}=-\frac{1}{q^{2n - 3}z_{n-1}^2}
$.

Odd $n$ corresponds to the diagram
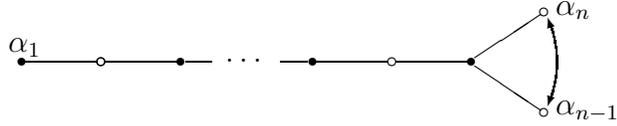
\begin{center}
\begin{picture}(220,50)

 \put(20,20){\circle*{3}}
\put(21.5,20){\line(1,0){27}}
\put(50,20){\circle{3}}
\put(51.5,20){\line(1,0){27}}

\put(82,20){\line(1,0){10}}
\put(118,20){\line(1,0){10}}
\put(80,20){\circle*{3}}
\put(97,20){$\ldots$}

\put(50,20){\circle{3}}
\put(51.5,20){\line(1,0){27}}

\put(130,20){\circle*{3}}
\put(131.5,20){\line(1,0){27}}

\put(160,20){\circle{3}}
\put(161.5,20){\line(1,0){27}}

\put(190,20){\circle*{3}}
\put(191.5,21.5){\line(3,2){25}}
\put(191.5,18.5){\line(3,-2){25}}
\put(217.5,39){\circle{3}}
\put(217.5,1){\circle{3}}

\put(15,24){$\al_1$}

\put(222,0){$\al_{n-1}$}
\put(222,38){$\al_n$}

\qbezier(220,6)(225,20)(220 ,34) \put(221,9){\vector(-1,-3){2}} \put(221,31){\vector(-1,3){2}}
 \end{picture}
\end{center}
We define
$$F_{\tilde \al_{n-1}}=[F_{n},F_{n-2}]_{q},\quad
F_{\tilde \al_n}=[F_{n-1},F_{n-2}]_{ q},\quad
\mbox{and} \quad
c_{\al_{n-1}}=-c_{\al_{n}}=
\frac{\sqrt{-1}}{q^{n-1}z_{n-1}}.
$$
Thus we have described $X_\al$ for all $\al$ that are not fixed by $\theta$.
Constructed this way, $U_q(\k)$ fixes the quantum point
(\ref{so(N)-Herm}).

\vspace{20pt}.

\noindent
\underline{\large \bf Acknowledgement.}

\vspace{10pt}
\noindent
This work is done at the Center of Pure Mathematics, MIPT, with
financial support of the project FSMG-2023-0013.

The first author (D.A.) is  thankful to the Deanship of Scientific Research at University of Bisha for the financial support through the Scholarship Program of the University.

\subsection*{Declarations}
\subsubsection*{Data Availability}
 Data sharing not applicable to this article as no datasets were generated or analysed during the current study.

\subsubsection*{Competing interests}
The authors have no competing interests to declare that are relevant to the content of this article.

 \end{document}